\documentclass[12pt,oneside,english]{amsart}
\usepackage{lmodern}

\usepackage{courier}
\usepackage[T1]{fontenc}
\usepackage[latin9]{inputenc}
\usepackage{listings}
\usepackage[letterpaper]{geometry}
\usepackage{prettyref}
\usepackage{float}
\usepackage{amsthm}
\usepackage{amstext}
\usepackage{amssymb}
\usepackage{graphicx}

\makeatletter

\floatstyle{ruled}
\newfloat{algorithm}{tbp}{loa}
\providecommand{\algorithmname}{Algorithm}
\floatname{algorithm}{\protect\algorithmname}

\numberwithin{equation}{section}
\numberwithin{figure}{section}
\theoremstyle{plain}
\newtheorem{thm}{\protect\theoremname}
  \theoremstyle{plain}
  \newtheorem{lem}[thm]{\protect\lemmaname}
  \theoremstyle{plain}
  \newtheorem{cor}[thm]{\protect\corollaryname}

\usepackage{algpseudocode}

\@ifundefined{showcaptionsetup}{}{%
 \PassOptionsToPackage{caption=false}{subfig}}
\usepackage{subfig}
\makeatother

\usepackage{babel}
  \providecommand{\corollaryname}{Corollary}
  \providecommand{\lemmaname}{Lemma}
\providecommand{\theoremname}{Theorem}

\begin{document}

\title[Finite Difference Weights]{Finite Difference Weights, Spectral Differentiation, and Superconvergence}

\author{Burhan Sadiq and Divakar Viswanath}

\thanks{NSF grant DMS-0715510 and SCREMS-1026317.}

\email{bsadiq@umich.edu and divakar@umich.edu}
\begin{abstract}
Let $z_{1},z_{2},\ldots,z_{N}$ be a sequence of distinct grid points.
A finite difference formula approximates the $m$-th derivative $f^{(m)}(0)$
as $\sum w_{k}f\left(z_{k}\right)$, with $w_{k}$ being the weights.
We derive an algorithm for finding the weights $w_{k}$ which is an
improvement of an algorithm of Fornberg (\emph{Mathematics of Computation},
vol. 51 (1988), p. 699-706). This algorithm uses fewer arithmetic
operations than that of Fornberg by a factor of $4/(5m+5)$ while
being equally accurate. The algorithm that we derive computes finite
difference weights accurately even when $m$, the order of the derivative,
is as high as $16$. In addition, the algorithm generalizes easily
to the efficient computation of spectral differentiation matrices. 

The order of accuracy of the finite difference formula for $f^{(m)}(0)$
with grid points $hz_{k}$, $1\leq k\leq N$, is typically $\mathcal{O}\left(h^{N-m}\right)$.
However, the most commonly used finite difference formulas have an
order of accuracy that is higher than the typical. For instance, the
centered difference approximation $\left(f(h)-2f(0)+f(-h)\right)/h^{2}$
to $f''(0)$ has an order of accuracy equal to $2$ not $1$ . Even
unsymmetric finite difference formulas can exhibit such superconvergence
or boosted order of accuracy, as shown by the explicit algebraic condition
that we derive. If the grid points are real, we prove a basic result
stating that the order of accuracy can never be boosted by more than
$1$. 
\end{abstract}
\maketitle

\section{Introduction}

Since the beginning of the subject, finite difference methods have
been widely used for the numerical solution of partial differential
equations. Finite difference methods are easier to implement than
finite element or spectral methods. For handling irregular domain
geometry, finite difference methods are better than spectral methods
but not as flexible as finite element discretizations. 

The basic problem in designing finite difference discretizations is
to approximate $f^{(m)}(0)$, the $m$-th derivative of the function
$f(z)$ at $z=0$, using function values at the grid points $hz_{1},hz_{2},\ldots,hz_{N}$.
The grid points can be taken as $z_{1},\ldots,z_{N}$ by setting the
mesh parameter $h=1$. We make the mesh parameter $h$ explicit where
necessary but suppress it otherwise. The finite difference formula
can be given as either 
\begin{equation}
f^{m}(0)\approx w_{1,m}f\left(z_{1}\right)+\cdots+w_{N,m}f\left(z_{N}\right)\label{eq:FD-sans-h}
\end{equation}
or
\begin{equation}
f^{(m)}\left(0\right)\approx\frac{w_{1,m}f\left(hz_{1}\right)+\cdots+w_{N,m}f\left(hz_{N}\right)}{h^{m}}.\label{eq:FD-with-h}
\end{equation}
If we require \prettyref{eq:FD-with-h} to have an error that is $\mathcal{O}\left(h^{N-m}\right)$
for smooth $f$, the choice of the weights $w_{k,m}$, $1\leq k\leq N$,
is unique (see Section 8). The grid points are always assumed to be
distinct.

Some finite difference formulas such as the centered difference approximations
to $f'(0)$ and $f''(0)$---$\left(f(h)-f(-h)\right)/2h$ and $\left(f(h)-2f(0)+f(-h)\right)/h^{2}$,
respectively---occur very commonly and are part of the bread and butter
of scientific computation. The most common finite difference formulas
presuppose an evenly spaced grid. However, evenly spaced grids are
often inadequate. In applications, it is frequently necessary to make
the grid finer within boundary layers or internal layers where the
underlying phenomenon is characterized by rapid changes. In addition,
evenly spaced grids do not lend themselves to adaptive mesh refinement.
For these reasons, it is often necessary to use grids that are not
evenly spaced. 

Fornberg \cite{Fornberg1988,Fornberg1998,Fornberg1998Book} devised
an algorithm for determining the weights $w_{i,m}$ given the grid
points $z_{k}$. The results of this paper include two algorithms
(see Sections 3 and 4) that improve Fornberg's. The numerical stability
of algorithms to find finite difference weights can be subtle. Therefore
we will begin this introduction by considering numerical stability.

\emph{Numerical stability.} All algorithms to compute finite difference
weights come down to multiplying binomials of the form $(z-z_{k})$
and extracting coefficients from the product. The numerical stability
of multiplying binomials, or equivalently of going from roots of a
polynomial to its coefficients, has aspects that are not obvious at
first sight. A dramatic example is the product $\left(z-\omega^{0}\right)\left(z-\omega^{1}\right)\ldots\left(z-\omega^{N-1}\right)$
where $\omega=\exp(2\pi i/N)$. Mathematically the answer is $z^{N}-1$.
Numerically the error is as high as $10^{15}$ for $N=128$ in double
precision arithmetic \cite{CalvettiReichel2003}. For numerical stability,
the binomials must be ordered using the bit reversed ordering or the
Leja ordering or some other scheme as shown by Calvetti and Reichel
\cite{CalvettiReichel2003}. The roots must be ordered in such a way
that the coefficients of intermediate products are not too large. 

Another point related to numerical stability comes up frequently.
Suppose we want to multiply $(z-\alpha)$ into the polynomial $\left(a_{0}+\cdots+a_{M}z^{M}+\cdots\right)$.
The obvious way to form the coefficients of the product, which is
denoted $b_{0}+\cdots+b_{M}z^{M}+\cdots$, is to use
\begin{eqnarray}
b_{0} & = & -\alpha a_{0}\nonumber \\
b_{j} & = & -\alpha a_{j}+a_{j-1}\quad\text{for }1\leq j\leq M.\label{eq:binom-poly-mult}
\end{eqnarray}

A related problem is to assume $a_{0}+a_{1}z+\cdots=(z-\alpha)^{-1}\left(b_{0}+b_{1}z+\cdots\right)$
and find the $a_{j}$ in terms of the $b_{j}$. The equations that
comprise \prettyref{eq:binom-poly-mult} can be easily inverted assuming
$\alpha\neq0$:
\begin{eqnarray}
a_{0} & = & -b_{0}/\alpha\nonumber \\
a_{j} & = & (a_{j-1}-b_{j})/\alpha\quad\text{for }j=1,2,\ldots\label{eq:inverse-poly-mult}
\end{eqnarray}
Our point is that \prettyref{eq:binom-poly-mult} appears to be safer
numerically than \prettyref{eq:inverse-poly-mult}. This point is
discussed further in Section 6, but we note here that the system of
equations \prettyref{eq:binom-poly-mult} does not involve back substitution
while the system of equations \prettyref{eq:inverse-poly-mult} involves
back substitution. We codify this observation as a rule of thumb.

\vskip .2cm \noindent \emph{Rule of thumb:} Algorithms that use triangular
systems of recurrences with back substitution, for example systems
such as \prettyref{eq:inverse-poly-mult}, tend to be numerically
unsafe.\vskip .2cm

This rule of thumb will guide our derivation of a numerically stable
algorithm for finding finite difference weights in Section 4. However,
there are exceptions to it as we will see. 

\emph{Computation of finite difference weights.}\textbf{ }There exists
a unique polynomial $\pi(z)$ of degree $N-1$ which satisfies the
interpolation conditions $\pi(z_{k})=f_{k}$ for $k=1\ldots N$ \cite{Davis1975}.
The Lagrange form of this interpolating polynomial is given by 
\begin{equation}
\pi(z)=\sum_{k=1}^{N}w_{k}\pi_{k}(z)f_{k}\:\:\text{where}\:\:\pi_{k}(z)=\prod_{j\neq k}(z-z_{j})\:\text{and}\: w_{k}=1/\pi_{k}(z_{k}).\label{eq:lagrange-interpolant}
\end{equation}
The finite difference weight $w_{k,m}$ is equal to the coefficient
of $z^{m}$ in $w_{k}\pi_{k}(z)$ times $m!$ (see Section 8). The
computation of the Lagrange weights $w_{k}$ takes $2N^{2}$ arithmetic
operations roughly half of which are multiplications and half are
additions or subtractions (all operation counts are given to leading
order only). 

In effect, Fornberg's algorithm \cite{Fornberg1988} is to multiply
the binomials $(z-z_{j})$ using recursion of the form \prettyref{eq:binom-poly-mult}
to determine the coefficient of $z^{m}$ in the Lagrange cardinal
function $w_{k}\pi_{k}(z_{k})$. The algorithm is not presented in
this way in \cite{Fornberg1988}. Instead it is organized to yield
the finite difference weights for partial lists of grid points $z_{1},\ldots,z_{k}$
with $k$ increasing from $1$ to $N$. Fornberg's algorithm requires
$5N^{2}/2+5MN^{2}/2$ arithmetic operations to determine the weights
$w_{k,m}$, $1\leq k\leq N$, if $m=M$. In the operation count, the
coefficient of $N^{2}$ is proportional to $M$ because each Lagrange
cardinal function $w_{k}\pi_{k}(z)$ is treated independently. Since
the order of the derivative goes up to $M$ only, coefficients beyond
the $z^{M}$ term are not needed and are not computed.

The algorithm for finding finite difference weights presented in Section
3 uses the modified Lagrange formula \cite{BerrutTrefethen2004}:
\begin{equation}
\pi(z)=(z-z_{1})\ldots(z-z_{N})\left(\frac{w_{1}}{z-z_{1}}f_{1}+\cdots+\frac{w_{N}}{z-z_{N}}f_{N}\right).\label{eq:mlagrange}
\end{equation}
Here the idea is to begin by determining the Lagrange weights $w_{k}$
and the coefficients (up to the $z^{M}$ term) of the polynomial $\pi^{\ast}(z)=\prod_{k=1}^{N}(z-z_{k})$
. The coefficients are determined using a recursion of the form \prettyref{eq:binom-poly-mult}
repeatedly. Since $w_{k}\pi_{k}(z)=w_{k}\pi^{\ast}(z)/(z-z_{k})$
, we may then use a recursion of the form \prettyref{eq:inverse-poly-mult}
to determine the finite difference weights. This algorithm uses $2N^{2}+6MN$
operations. The $2N^{2}$ term is the expense of computing the Lagrange
weights $w_{k}$. 

Since the method based on the modified Lagrange formula uses \prettyref{eq:inverse-poly-mult},
it is numerically unsafe according to our rule of thumb. Indeed, for
large $M$ it is not numerically stable. The computations of Section
7 suggest that the method based on the modified Lagrange formula is
a good choice for $M\leq4$ but not for larger $M$. 

In Section 4, we derive another algorithm, one based on the following
partial products:
\[
l_{k}(z)=\prod_{j=1}^{k}(z-z_{j})\:\:\text{and}\:\: r_{k}(z)=\prod_{j=k}^{N}(z-z_{j}).
\]
By convention, $l_{0}\equiv r_{N+1}\equiv1$. For $k=1,\ldots,N$,
coefficients of these partial products are computed up to the $z^{M}$
term using recursions of the form \prettyref{eq:binom-poly-mult}
repeatedly. Since $\pi_{k}=l_{k-1}r_{k+1}$, the finite difference
weights $w_{k,m}$ for $m=0,\ldots,M$, are obtained by convolving
the coefficients of $l_{k-1}$ and $r_{k+1}$ followed by a multiplication
by $m!$ and the Lagrange weight $w_{k}$. This algorithm uses $2N^{2}+6NM+NM^{2}$
arithmetic operation. 

Even though the method based on partial products has an additional
expense of $NM^{2}$ operations, we recommend it for all uses. Because
it completely avoids back substitution, it has good numerical stability.
If FFTs are used for convolution, the $NM^{2}$ term can be replaced
by $\mathcal{O}\left(NM\log M\right)$. Instances where $M$ is so
large that the use of FFTs is advantageous are unlikely to occur in
practice.

\emph{Spectral differentiation.}\textbf{ }Beginning with \prettyref{eq:FD-sans-h},
our discussion of finite difference weights has assumed $z=0$ to
be the point of differentiation. Given grid points $z_{1},\ldots,z_{N}$,
the spectral differentiation matrix of order $M$ is an $N\times N$
matrix. If the $(i,j)$-th entry is denoted $\omega_{i,j}$, then
$f^{(M)}(z_{i})\approx\sum_{j}\omega_{i,j}f(z_{j})$. Thus $\omega_{i,j}$
is the finite difference weight at $z_{j}$ if the point of differentiation
is $z=z_{i}$. The point of differentiation can be shifted to $0$
by replacing the grid points $z_{1},\ldots,z_{N}$ by $z_{1}-z_{i},\ldots,z_{N}-z_{i}$
.

Applying Fornberg's method row by row would cost $\mathcal{O}\left(N^{3}\right)$
arithmetic operations. Welfert \cite{Welfert1997} modified Fornberg's
recurrences and obtained a method that computes the spectral differentiation
matrix of order $M$ using only $\mathcal{O}\left(N^{2}M\right)$
operations. 

The way to modify the first of our two methods (Section 3) so as to
compute spectral differentiation matrices in $\mathcal{O}\left(N^{2}M\right)$
arithmetic operations is almost obvious. We simply have to note that
the Lagrange weights $w_{k}$ defined by \prettyref{eq:lagrange-interpolant}
do not change at all when the entire grid is shifted. Therefore Lagrange
weights are the same for every row of the spectral differentiation
matrix and need to be computed just once using $\mathcal{O}\left(N^{2}\right)$
operations. The rest of the computation is repeated for every row,
with an appropriately shifted grid, costing $\mathcal{O}\left(N^{2}M\right)$
operations to determine the entire spectral differentiation matrix.

If the method based on partial products (Section 4) is used to determine
the coefficients of the Lagrange cardinal functions $w_{k}\pi_{k}(z)$,
the cost of computing the spectral differentiation matrix is $\mathcal{O}\left(N^{2}M^{2}\right)$.
Although Welfert's method and the method of Section 3 compute spectral
differentiation matrices with a lower asymptotic cost, they are less
accurate than the method based on partial products. They should not
be used for $M>4$. Welfert \cite{Welfert1997} stated that the problem
of round-off errors becomes {}``very important'' for $M>6$ and
that his tables did not use a large enough $M$ to expose the problem.
In contrast, the method based on partial products computes every entry
of the $512\times512$ Chebyshev differentiation matrix for the $16$-th
derivative with $9$ or more digits of accuracy.

\emph{Superconvergence or boosted order of accuracy.}\textbf{ }If
the number of grid points is $N$ and the order of the derivative
is $m$, the finite difference weights are unique if the difference
formula is required to be $\mathcal{O}\left(h^{N-m}\right)$ (see
Section 8). For certain grids, these unique weights imply an error
of $\mathcal{O}\left(h^{N-m+1}\right)$, which is of higher order
than what is typical for $N$ grid points and the $m$-th derivative.
We term this as superconvergence or boosted order of accuracy. In
section 8, we give explicit conditions for boosted order of accuracy.
The finite difference approximation \prettyref{eq:FD-with-h} to $f^{(m)}(0)$
has an order of accuracy boosted by $1$ if and only if
\[
S_{N-m}=0
\]
 where $S_{k}$ is the elementary symmetric function
\[
\sum_{1\leq i_{1}<\cdots<i_{k}\leq N}z_{i_{1}\ldots}z_{i_{k}}.
\]
 If the grid points are real, we prove that the order of accuracy
cannot be boosted by more than $1$.

For the special case $m=2$, the finite difference approximation to
the second derivative at $z=0$ using three grid points has an order
of accuracy equal to $2$, which is a boost of $1$, if and only if
the grid points satisfy $z_{1}+z_{2}+z_{3}=0$. Evidently, this condition
is satisfied by the grid points $-1,0,1$ used by the centered difference
formula. An unsymmetric choice of grid points such as $-3,1,2$ also
boosts the order of accuracy by $1$. However, no choice of $z_{1}$,
$z_{2}$, and $z_{3}$ on the real line can boost the order of accuracy
by more than 1. 

With four grid points and $m=2$, the condition for a boost in the
order of accuracy is 
\[
z_{1}z_{2}+z_{1}z_{3}+z_{1}z_{4}+z_{2}z_{3}+z_{2}z_{4}+z_{3}z_{4}=0.
\]
No choice of $z_{1}$, $z_{2}$, $z_{3}$, and $z_{3}$ on the real
line can boost the order of accuracy of the finite difference approximation
to $f''(0)$ by more than 1. The maximum possible order of accuracy
is $3$. In this case, a symmetric choice of grid points such as $-2,-1,1,2$
does not boost the order of accuracy. Unsymmetric grid points that
boost the order of accuracy to $3$ can be found easily. For example,
the order of accuracy is $3$ for the grid points $-2/3,0,1,2.$ 

These results about superconvergence or boosted order of accuracy
of finite difference formulas are quite basic. It is natural to suspect
that they may have been discovered a long time ago. However, the results
are neither stated nor proved in any source that we know of.

If the grid points $z_{i}$ are allowed to be complex, the order of
accuracy can be boosted further but not by more than $m$. The order
of accuracy is boosted by $k$ with $1\leq k\leq m$ if and only if
\[
S_{N-m}=S_{N-m+1}=\cdots=S_{N-m+k-1}=0.
\]
An algorithm to detect the order of accuracy and compute the error
constant of the finite difference formula \prettyref{eq:FD-with-h}
is given in section 8.

\section{From roots to coefficients}

Given $\alpha_{1},\ldots,\alpha_{N}$, the problem is to determine
the coefficients of a polynomial of degree $N$ whose roots are $\alpha_{1},\ldots,\alpha_{N}$.
The polynomial is evidently given by $\prod_{k=1}^{N}(z-\alpha_{k})$.
If the product $\prod_{k=1}^{n}(z-\alpha_{k})$ is given by $c_{0}+c_{1}z+\cdots+z^{n}$,
then the coefficients $c_{0}',\, c_{1}',\ldots$ of the product $\prod_{k=1}^{n+1}(z-\alpha_{k})$
are formed using $c_{0}'=-c_{0}\alpha_{n+1}$ and $c_{m}'=-c_{m}\alpha_{n+1}+c_{m-1}$
for $m=\mbox{1},2,\ldots$ as in \prettyref{eq:binom-poly-mult}.

For accurate evaluation of the coefficients, Calvetti and Reichel
\cite{CalvettiReichel2003} have demonstrated the need to order the
roots $\alpha_{k}$ carefully. The roots $\alpha_{k}$ must be ordered
in such a way that the coefficients of the partial products $\prod_{k=1}^{n}(z-\alpha_{k})$,
$n=1,\ldots,N-1$, that occur in intermediate stages are not much
bigger than the coefficients of the complete product. If $\alpha_{k}=\omega^{k-1}$,
where $\omega=\exp(2\pi i/N)$, the natural ordering $\omega^{0},\omega^{1},\ldots,\omega^{N-1}$
is numerically unsound. If $N$ is large the first roots in this sequence
are close to $1$ leading to partial products which resemble $(z-1)^{n}$
and have coefficients that are of the order of the binomial coefficients.
In contrast, the complete product is simply $z^{N}-1$.

This matter of ordering the roots carefully is equivalent to choosing
a good order of grid points when determining finite difference weights.
Good ordering of grid points may improve accuracy but is not as important
as it is in the general problem of determining coefficients from roots.
When determining finite difference weights for a derivative of order
$M$, we need coefficients of terms $1,z,\ldots,z^{M}$ but no higher.
The most dramatic numerical stabilities in determining coefficients
occur near the middle of the polynomial, but $M$, which is the order
of differentiation, will not be large in the determination of finite
difference weights. 

Suppose that the $\alpha_{i}$ are all nonzero and that $\prod_{k=1}^{N}(z-\alpha_{k})=c_{0}+\cdots+c_{M}z^{M}+\mathcal{O}\left(z^{M+1}\right).$
Another algorithm to compute $c_{0},\ldots,c_{M}$ is obtained as
follows. Let 
\begin{eqnarray*}
\mathcal{P}_{r} & = & \sum_{k=1}^{N}\alpha_{k}^{-r}\\
\mathcal{E}_{r} & = & \sum_{1\leq i_{1}<\cdots<i_{r}\leq N}\left(\alpha_{i_{1}}\ldots\alpha_{i_{r}}\right)^{-1}.
\end{eqnarray*}
By the Newton identities
\begin{eqnarray*}
\mathcal{E}_{1} & = & \mathcal{P}_{1}\\
2\mathcal{E}_{2} & = & \mathcal{E}_{1}\mathcal{P}_{1}-\mathcal{P}_{2}\\
3\mathcal{E}_{3} & = & \mathcal{E}_{2}\mathcal{P}_{1}-\mathcal{E}_{1}\mathcal{P}_{2}+\mathcal{P}_{3}
\end{eqnarray*}
and so on. By convention, $\mathcal{E}_{0}=1$. The algorithm begins
by computing the power sums $\mathcal{P}_{1},\ldots,\mathcal{P}_{M}$
directly and uses the Newton identities to compute the elementary
symmetric functions $\mathcal{E}_{r}$, $0\leq r\leq M$. The coefficients
are obtained using 
\[
c_{r}=(-1)^{N+r-1}\mathcal{E}_{r}\prod_{k=1}^{N}\alpha_{k}.
\]
This algorithm does not really presuppose an ordering of the $\alpha_{k}$
and the computation of the power sums $\mathcal{P}_{r}$ is backward
stable, and especially so if compensated summation is used \cite{Higham2002}.
If this method is used to compute the product $\prod_{k=1}^{N}\left(z-\omega^{k-1}\right)$,
where $\omega$ is as before, it finds the coefficients of the product
with excellent accuracy. But in general this method is inferior to
the repeated use of \prettyref{eq:binom-poly-mult} after choosing
a good ordering of the roots $\alpha_{k}$. By way of a partial explanation,
we note that the Newton identities have a triangular structure with
back substitution, which is deemed to be possibly unsound by the rule
of thumb stated in the introduction.

\section{Finite difference weights using the modified lagrange formula}

Let the grid points be $z_{1},\ldots,z_{N}$ with $f_{1},\ldots,f_{N}$
being the function values at the grid points. Define 
\begin{equation}
\pi_{k}(z)=\prod_{j\neq k}^{N}(z-z_{j}).\label{eq:pik}
\end{equation}
Then the Lagrange interpolant shown in \prettyref{eq:lagrange-interpolant}
is $\pi(z)=\sum_{k=1}^{N}w_{k}\pi_{k}(z)f_{k}.$ The weights $w_{k}$
equal $1/\pi_{k}(z_{k})$. Our objective is to derive formulas for
$d^{m}\pi(z)/dz^{m}$ at $z=0$ for $m=1,\ldots,M$. The $m=0$ case
is regular Lagrange interpolation. The weights $w_{k}$ will be assumed
to be known. The formulas for $d^{m}\pi(z)/dz^{m}$ at $z=0$ will
be linear combinations of $f_{k}$ with weights. We assume $1\leq M\leq N-1$.

If the coefficient of $z^{m}$ in $\pi_{k}(z)$ is denoted by $c_{k,m}$,
we have
\[
\frac{d^{m}\pi(z)}{dz^{m}}\Biggl|_{z=0}=m!\,\sum_{k=1}^{N}c_{k,m}w_{k}f_{k}.
\]
The finite difference weights are then given by 
\begin{equation}
w_{k,m}=m!w_{k}c_{k,m}.\label{eq:wkm-ckm}
\end{equation}
Once the $c_{k,m}$ are known, the weights $w_{k,m}$ are computed
using \prettyref{eq:wkm-ckm} for $k=1,\ldots,N$ and $m=1,\ldots,M$. 

Let $\pi^{\ast}(z)$ denote the polynomial $\prod_{k=1}^{N}(z-z_{k})$.
Let

\[
\pi^{\ast}(z)=\sum_{k=0}^{N}C_{k}z^{k}.
\]
 Notice that $\pi^{\ast}(z)$ occurs as a factor in front of the modified
Lagrange formula \prettyref{eq:mlagrange}. Our method for calculating
$w_{k,m}$ begins by calculating $C_{0},C_{1},\ldots,C_{M+1}$. The
$c_{k.m}$ are determined using $C_{0},\ldots,C_{M+1}$ and then the
$w_{k,m}$ are determined using \prettyref{eq:wkm-ckm}. The Lagrange
weights figure in this last step. It is in this sense that the method
for determining finite difference weights described in this section
uses the modified Lagrange formula.

We start by setting $C_{0}=1$ and $C_{1}=\cdots=C_{M+1}=0$. For
each $j=1,2,\ldots,N$ , the following update is performed:
\begin{eqnarray*}
C_{0}' & = & -z_{j}C_{0}\\
C_{1}' & = & -z_{j}C_{1}+C_{0}\\
 & \cdots\\
C_{M}' & = & -z_{j}C_{M}+C_{M-1}\\
C_{M+1}' & = & -z_{j}C_{M+1}+C_{M}
\end{eqnarray*}
followed by $C_{0}=C_{0}',\ldots,C_{M+1}=C_{M+1}'$. 

To obtain the $c_{k,m}$, use $\left(z-z_{k}\right)\pi_{k}(z)=\pi^{\ast}(z)$
to get
\begin{eqnarray*}
-z_{k}c_{k,0} & = & C_{0}\\
-z_{k}c_{k,1}+c_{k,0} & = & C_{1}\\
 & \cdots\\
-z_{k}c_{k,M}+c_{k,M-1} & = & C_{M}\\
-z_{k}c_{k,M+1}+c_{k,M} & = & C_{M+1}.
\end{eqnarray*}
If $z_{k}\neq0$, we have
\begin{eqnarray}
c_{k,0} & = & -C_{0}/z_{k}\nonumber \\
c_{k,1} & = & \left(c_{k,0}-C_{1}\right)/z_{k}\nonumber \\
 & \cdots\nonumber \\
c_{k,M} & = & \left(c_{k,M-1}-C_{M}\right)/z_{k}.\label{eq:ckm-back-subs}
\end{eqnarray}
If $z_{k}=0$, we have $c_{k,m}=C_{m+1}$ for $m=0,1,\ldots,M$. We
can now use \prettyref{eq:wkm-ckm} to find the finite difference
weights. The complete algorithm is exhibited as Algorithm \ref{alg:fdweights-mlgrng}.

\begin{algorithm}
\begin{algorithmic}[1]
\Function{LagrangeWeights}{$z_1,\ldots,z_N$,$w_1,\ldots,w_N$}
\For{$i=1,2,\ldots,N$}
\State $w_i=\prod_j(z_i-z_j)$ over $j=1,\ldots, N$ but $j\neq i$.
\State $w_i = 1/w_i$
\EndFor
\EndFunction
\Function{FindC}{$z_1,\ldots,z_N$, $C_0,\ldots,C_{M+1}$}
\State{Temporaries: $t_0,\ldots,t_{M+1}$}
\State{$C_0=1$ and $C_i=0$ for $1\leq i\leq M+1$}
\For{$j=1,\ldots,N$}
\State{$t_0 = -z_j C_0$}
\State{$t_i = C_{i-1}-z_j C_i$ for $i=1,2,\ldots,M+1$}
\State{$C_i=t_i$ for $i=0,1,\ldots,M+1$}
\EndFor
\EndFunction
\Function{findckm}{$z_k$, $C_0,C_1,\ldots,C_{M+1}$, $c_{k,0},\ldots,c_{k,M}$}
\If{$z_k==0$}
\State{$c_{k,m}=C_{m+1}$ for $m=0,\ldots,M$}
\Else
\State{$\zeta = 1/z_k$}
\State{$c_{k,0}=-\zeta C_0$}
\State{$c_{k,m}=\zeta(c_{k,m-1}-C_k)$ for $m=1,\ldots,M$}
\EndIf
\EndFunction
\Function{FindWeights}{$c_{k,0},\ldots,c_{k,M}$, $w_k$, $w_{k,0},\ldots,w_{k,M}$}
\State $f = w_k$
\For{$m=0,1,\ldots,M$}
\State $w_{k,m} = f c_{k,m}$
\State $f = (m+1)f$
\EndFor
\EndFunction
\Function{FindAllWeights}{$z_1,\ldots,z_N$, $w_{k,m}$ for $k=1,\ldots,N$ and $m=0,\ldots,M$}
\State Temporaries: $w_1,\ldots,w_N$
\State \Call{LagrangeWeights}{$z_1,\ldots,z_N$,$w_1,\ldots,w_N$}
\State Temporaries: $C_0,\ldots,C_{M+1}$
\State \Call{FindC}{$z_1,\ldots,z_N$, $C_0, \ldots, C_{M+1}$}
\State Temporaries: $c_{k,m}$ for $k=1,\ldots,N$ and $m=0,\ldots,M$
\For{$k=1,\ldots,N$}
\State\Call{findckm}{$z_k$, $C_0,C_1,\ldots,C_{M+1}$, $c_{k,0},\ldots,c_{k,M}$}
\State\Call{FindWeights}{$c_{k,0},\ldots,c_{k,M}$, $w_k$, $w_{k,0},\ldots,w_{k,M}$}
\EndFor
\EndFunction
\end{algorithmic}\caption{Finite difference weights using the modified Lagrange formula \label{alg:fdweights-mlgrng}}
\end{algorithm}
The operation count for Algorithm \ref{alg:fdweights-mlgrng} is as
follows. The operation counts are given to leading order only.
\begin{itemize}
\item The function L{\scriptsize AGRANGE}W{\scriptsize EIGHTS()} invoked
on line 34 computes the $w_{k}$ using $2N^{2}$ operations. More
precisely, the operations are $N(N-1)$ subtractions, $N(N-2)$ multiplications,
and $N$ divisions. The number of subtractions can be halved using
additional storage.
\item The function F{\scriptsize IND}C() invoked on line 36 computes $C_{0},\ldots,C_{M+1}$.
The number of operations used is $2MN$. More precisely, the operations
are $N(M+2)$ multiplications and $N(M+1)$ additions.
\item The function {\scriptsize FINDCKM}() invoked $N$ times on line 39
computes $c_{k,m}$. The number of operations used is $2MN$. More
precisely, the operations are $N(M+1)$ multiplications, $NM$ additions,
and one division. 
\item The function F{\scriptsize IND}W{\scriptsize EIGHTS}() invoked $N$
times on line 40 computes the finite difference weights $w_{k,m}$.
The number of operations used is $2MN$. This function implements
$w_{k,m}=m!w_{k}c_{k,m}$, which is \prettyref{eq:wkm-ckm}, using
two multiplications for each $w_{k,m}$. Even if the factorials $m!$
are precomputed and stored, we need the same number of multiplications
for each $w_{k,m}$
\end{itemize}
The total number of floating point operations is $2N^{2}+6MN$. 

The operation count of Fornberg's method is $5N^{2}/2+5MN^{2}/2-5M^{3}/6$
assuming $M\ll N$. Our algorithm differs from that of Fornberg \cite{Fornberg1988}
in two major respects. Firstly, Fornberg does not compute the Lagrange
weights $w_{k}$ explicitly as we do. Secondly, we form the coefficients
of $\pi^{\ast}(z)$ and use that to recover the coefficients of $\pi^{k}(z)$
for $k=1,\ldots,N$. Fornberg's method is laid out quite differently
from ours, but in effect it treats each $\pi_{k}(z)$ separately. 

Because Fornberg's method builds up the finite difference weights
for the grid $z_{1},\ldots,z_{N}$ using the finite difference weights
of the partial grids $z_{1},\ldots,z_{k}$, with $k$ increasing from
$1$ to $N$, it is forced to use $\mathcal{O}\left(N^{2}\right)$
divisions. Algorithm \prettyref{alg:fdweights-mlgrng} uses only $2N$
divisions. These occur in the computation of the Lagrange weights
(line 4) and in determining $c_{k,m}$ (line 20). Similarly, Algorithm
\prettyref{alg:fdweights-partial-prod}, which is derived in the next
section, uses only $N$ divisions. On current processors, division
is more expensive than multiplications or additions. For example,
in the Intel Nehalem microarchitecture, the latency of division is
three to six times that of multiplication. While multiplication instructions
can be dispatched to ports in successive clock cycles, the dispatch
of division instructions must be separated by five or so clock cycles. 

Many, if not most, of the numerical analysis textbooks recommend the
Newton form for polynomial interpolation. The weights of the Newton
form can be computed with $\mathcal{O}\left(N^{2}\right)$ arithmetic
operations using the divided differences table and updated using $\mathcal{O}(N)$
operations if a new grid point $z_{N+1}$ is added. In addition, the
Newton form can be evaluated at a point $z$ using $\mathcal{O}(N)$
operations. It has been well known that the weights of the Lagrange
form can be computed with $\mathcal{O}\left(N^{2}\right)$ operations,
but there was much less clarity about the evaluation of the Lagrange
form and the cost of updating the weights when a new grid point $z_{N+1}$
is added. In an engaging paper, Berrut and Trefethen \cite{BerrutTrefethen2004}
pointed out that an examination of \prettyref{eq:mlagrange} and the
formula \prettyref{eq:lagrange-interpolant} for the Lagrange weights
$w_{k}$ clarifies the updating and evaluation of the Lagrange form
to be as efficient as in the Newton case. For the use of the Lagrange
form for finding roots of functions, see Corless and Watt \cite{CorlessWatt2004}.

\section{Finite difference weights using partial products}

As we will see in Section 7, Algorithm \ref{alg:fdweights-mlgrng}
is accurate enough if $M\leq4$, but it should not be used if the
order of the derivative is higher than $4$. The problem is the use
of \prettyref{eq:ckm-back-subs} by the function \texttt{\footnotesize FINDCKM()}
to determine $c_{k,m}$. This step involves back substitution. We
will now derive an algorithm that completely avoids back substitution.

Let $l_{k}(z)=\prod_{j=1}^{k}(z-z_{j})$ and $r_{k}(z)=\prod_{j=k}^{N}(z-z_{j})$.
Denote the coefficients of $1,z,\ldots,z^{M}$ in $l_{k}(z)$ and
$r_{k}(z)$ by $L_{k,0},\ldots,L_{k,M}$ and $R_{k,0},\ldots,R_{k,M}$,
respectively. The coefficients $L_{k,m}$ are computed in the order
$k=1,2,\ldots,N$. The coefficients $R_{k,m}$ are computed in the
reverse order, which is $k=N,N-1,\ldots,1$. It is evident that $\pi_{k}(z)$,
which is defined by \prettyref{eq:lagrange-interpolant} or \prettyref{eq:pik},
is equal to $l_{k-1}(z)r_{k+1}(z)$. Therefore the coefficient $c_{k,m}$
of $z^{m}$ in $\pi_{k}(z)$ can be obtained using 
\[
c_{k,m}=\sum_{s=0}^{m}L_{k-1,m-s}R_{k+1,s}.
\]
The finite difference weight $w_{k,m}$ is obtained as $m!w_{k}c_{k,m}$,
where $w_{k}$ is the Lagrange weight at $z_{k}$.

Because this method stores $L_{k,m}$, $R_{k,m}$, and other intermediate
quantities, it is more convenient to implement it as a class than
in purely functional form. The public members of the C++ class are
shown below:{\footnotesize }
\begin{lstlisting}[basicstyle={\footnotesize},language={C++}]
class FDWeights{
  FDWeights(const double *zz, int NN, int MM);
  ~FDWeights();
  void setz0(double z0); //derivative at z0
  void setzk(int k0); //derivative at k0-th grid point
  double operator()(int m,  int k){//weight for mth derv at kth grid point
    assert((0<=m)&&(m<=M)&&(0<=k)&&(k<N));
    return fdw[k*(M+1)+m];
  }
  double operator()(int k){//weight for **Mth** derv at kth grid point
    assert((0<=k)&&(k<N));
    return fdw[k*(M+1)+M];
  }
};
\end{lstlisting}
The \texttt{fdw} array stores the finite difference weights so that
\texttt{fdw{[}k{*}(M+1)+m{]}} is equal to $w_{k,m}$. The implementation
of the member functions is displayed as Algorithm \ref{alg:fdweights-partial-prod}
.

\begin{algorithm}
\begin{algorithmic}[1]
\Function{FDWeights::FDWeights}{$z_1,\ldots,z_N$,$M$}
\State{Variables internal to class}
\State{(1): Order of derivative $M$ (initialized from argument list)}
\State{(2): Grid points $z_1,\ldots,z_N$ (initialized from argument list)}
\State{(3): Lagrange weights $w_1,\ldots,w_N$}
\State{(4): Partial product coefficients: $L_{k,m}$ and $R_{k,m}$ for $0\leq k\leq N+1$, $0\leq m \leq M$}
\State{(5): Finite difference weights: $w_{k,m}$ for $1\leq k\leq N$ and $0\leq m \leq M$}
\State\Call{LagrangeWeights}{$z_1,\ldots,z_N$,$w_1,\ldots,w_N$}
\State\Call{setz0}{0} (derivatives at $z=0$ by default)
\EndFunction
\Function{FDWeights::-FDWeights}{}
\State Deallocate all internal variables
\EndFunction
\Function{multbinom}{$a_0,\ldots,a_M$,$b_0,\ldots,b_M$,$\zeta$}
\State $b_0=-\zeta\,a_0$
\State $b_m=-\zeta\,a_m+a_{m-1}$ for $k=1,\ldots,M$
\EndFunction
\Function{convolve}{$a_0,\ldots,a_M$,$b_0,\ldots,b_M$,$c_0,\ldots,c_M$}
\State $c_m=a_mb_0+a_{m-1}b_1+\cdots+a_0b_m$ for $m=0,\ldots,M$
\EndFunction
\Function{FDWeights::setz0}{$\zeta$}
\State Temporaries: $\zeta_1,\ldots,\zeta_N$
\State $\zeta_k=z_k-\zeta$ for $k=1,\ldots,N$
\State $L_{0,m}=1$ for $m=0$ and $L_{k,m}=0$ for $m=1,\ldots,M$
\For{$k=1,\ldots,N$}
\State\Call{multbinom}{$L_{k-1,0},\ldots,L_{k-1,M}$,$L_{k,0},\ldots,L_{k,M}$, $\zeta_k$}
\EndFor
\State $R_{N+1,m}=1$ for $m=0$ and $R_{N+1,m}=0$ for $m=1,\ldots,M$.
\For{$k=N,N-1,\ldots,1$}
\State\Call{multbinom}{$R_{k+1,0},\ldots,R_{k+1,M}$,$R_{k,0},\ldots,R_{k,M}$,$\zeta_k$}
\EndFor
\For{$k=1,\ldots,N$}
\State Temporaries: $c_{k,m}$
\State\Call{convolve}{$L_{k-1,0},\ldots,L_{k-1,M}$,$R_{k+1,0},\ldots,R_{k+1,M}$,$c_{k,0},\ldots,c_{k,M}$}
\State\Call{FindWeights}{$c_{k,0},\ldots,c_{k,M}$,$w_k$,$w_{k,0},\ldots,w_{k,m}$}
\EndFor
\EndFunction
\Function{setk}{k}
\State \Call{setz0}{$z_k$}
\EndFunction
\end{algorithmic}

\caption{Finite difference weights using partial products. \label{alg:fdweights-partial-prod}}

\end{algorithm}

Algorithm \ref{alg:fdweights-partial-prod} invokes functions defined
as a part of Algorithm \ref{alg:fdweights-mlgrng} on lines 8 and
35. The total expense is $2N^{2}+6NM+NM^{2}$ arithmetic operations.

\section{Spectral differentiation matrices}

In Algorithm \ref{alg:fdweights-partial-prod}, the Lagrange weights
$w_{k}$ are computed in the class constructor \texttt{\footnotesize FDWeights::FDWeights()}.
If we want the finite difference weights for derivatives evaluated
at $z=\zeta$, we need to invoke the member function \texttt{\footnotesize FDWeights::setz0()}
with the argument $\zeta$. If we want the finite difference weights
for the derivatives evaluated at the $k$-th grid point $z_{k}$,
we need to invoke the member function \texttt{\footnotesize FDWeights::setzk()}.
The Lagrange weights are not re-computed when either of these member
functions is used to set the point at which derivatives are taken.
This has implications for spectral differentiation.

Suppose the grid to be $z_{1},\ldots,z_{N}$ as usual. In the $N\times N$
spectral differentiation matrix of order $M$, the $(i,j)$-th entry
$\omega_{i,j}$ is equal to the finite difference weight at $z_{j}$
when the derivative of order $M$ is taken at $z=z_{i}$. The spectral
differentiation matrix is computed in the following steps. 

\begin{algorithmic}[1]
\State FDWeights fd($z_1,\ldots,z_N$,$M$)
\For{$i=1,\ldots,N$}
\State fd.setzk($i$)
\State $\omega_{i,j}$=fd($j$) for $j=1,\ldots,N$
\EndFor
\end{algorithmic}In the class definition of the previous section, the function call
operator has been overloaded so that \texttt{fd(j)} returns the finite
difference weight at $z_{j}$ for the $M$-th derivative.

Since the Lagrange weights are computed just once, the cost of computing
the spectral differentiation matrix of order $M$ is $2N^{2}+6N^{2}M+N^{2}M^{2}$
arithmetic operations.

\section{Discussion of numerical stability}

Suppose that the sequence $a_{0},a_{1},\ldots$ and the sequence $b_{0},b_{1},\ldots$
are related by 
\[
\left(a_{0}+a_{1}z+a_{2}z^{2}+\cdots\right)=(z-\alpha)^{-1}\left(b_{0}+b_{1}z+b_{2}z^{2}+\cdots\right).
\]
Then the relationship between the sequences can be expressed in matrix
notation as {\footnotesize 
\begin{equation}
\left(\begin{array}{cccc}
-\alpha\\
1 & -\alpha\\
 & \ddots & \ddots\\
 &  & 1 & -\alpha
\end{array}\right)\left(\begin{array}{c}
a_{0}\\
a_{1}\\
\vdots\\
a_{M}
\end{array}\right)=\left(\begin{array}{c}
b_{0}\\
b_{1}\\
\vdots\\
b_{M}
\end{array}\right)\:\text{or}\:\left(\begin{array}{cccc}
-\frac{1}{\alpha}\\
-\frac{1}{\alpha^{2}} & -\frac{1}{\alpha}\\
\ddots & \ddots & \ddots\\
-\frac{1}{\alpha^{M+1}} & \ddots & -\frac{1}{\alpha^{2}} & -\frac{1}{\alpha}
\end{array}\right)\left(\begin{array}{c}
b_{0}\\
b_{1}\\
\vdots\\
b_{M}
\end{array}\right)=\left(\begin{array}{c}
a_{0}\\
a_{1}\\
\vdots\\
a_{M}
\end{array}\right).\label{eq:triangular-system}
\end{equation}
}The matrices that occur here will be denoted by $T$ and $T^{-1}$,
respectively. To calculate the $a_{j}$ given the $b_{j}$, the method
used ordinarily is back substitution---$a_{0}=-b_{0}/\alpha$ and
$a_{j}=(a_{j-1}-b_{j})/\alpha$ for $j\geq1$---as in \prettyref{eq:inverse-poly-mult}.
However, in Section 1, we stated as a rule of thumb that an intermediate
step which uses back substitution of that type is likely to be numerically
unsafe. In particular, that rule of thumb leads us to expect that
Algorithm \ref{alg:fdweights-mlgrng}, which uses back substitution
(to compute $c_{k,m}$), is possibly inferior to Algorithm \ref{alg:fdweights-partial-prod},
which does not use back substitution. In the next section, we will
show that that surmise is indeed true. Here we will discuss the rule
of thumb.

For convenience, we denote the two vectors that occur in \prettyref{eq:triangular-system}
as $\mathbf{a}$ and $\mathbf{b}$. If $\mathbf{a}$ is calculated
using back substitution, the following norm-wise bound on errors applies
\cite[Chapter 7]{Higham2002}: 
\[
\frac{||\tilde{\mathbf{a}}-\mathbf{a}||}{||\mathbf{a}||}\leq\frac{2\epsilon\kappa(T)}{1-\epsilon\kappa(T)}.
\]
Here $\tilde{\mathbf{a}}$ is the computed vector, $\epsilon$ is
a small multiple of the machine epsilon, and $\kappa(T)=||T||\,||T^{-1}||$
is the condition number. The norm can be any matrix norm such as the
$2$-norm or the $\infty$-norm. 

By inspecting $T$ and $T^{-1}$ displayed in \prettyref{eq:triangular-system},
it is evident that $\kappa(T)$ increases exponentially with $M$
if $|\alpha|<1$. The norm-wise bound suggests that the norm-wise
error in $\tilde{\mathbf{a}}$ increases exponentially with $M$. 

If $|\alpha|\geq1$, the matrix $T$ is evidently well-conditioned.
So it may appear as if the problem can be cured by recasting it to
make $|\alpha|\geq1$. Such a recasting is easy to accomplish. If
we write $(z-\alpha)=d\left(z/d-\alpha/d\right)$ for some $d\leq|\alpha|$
and expand the series in powers of $(z/d)$, then $\alpha/d$ will
replace $\alpha$ in the triangular systems and $T$ will therefore
be well-conditioned. Such scaling does not improve accuracy, however.
It is true that the norm-wise errors in the computed vector will be
small, but the entries of the computed vector will be poorly scaled.
When the computed vector is multiplied by inverse powers of $d$ to
recover entries of $\mathbf{a}$, the relative errors in entries such
as $a_{M}$ can be quite large.

Since $T$ is triangular, a component-wise bound of the following
type applies \cite[Chapter 8]{Higham2002}:
\[
\frac{||\tilde{\mathbf{a}}-\mathbf{a}||_{\infty}}{||\mathbf{a}||_{\infty}}\leq\frac{\text{cond}(T,\mathbf{a})\gamma_{n}}{1-\text{cond}(T)\gamma_{n}},\:\text{cond}(T,x)=\frac{||\,|T^{-1}|\,|T|\,|x|\,||_{\infty}}{||x||_{\infty}},\:\text{cond}(T,x)=||\,|T|\,|T^{-1}|\,||_{\infty}.
\]
Here $\gamma_{n}$ is approximately $2n$ times the machine epsilon
and $|T|$ is $T$ with its entries replaced by their absolute values.
This component-wise bound does not help much. If $|\alpha|<1$, the
condition numbers that occur in this bound again increase exponentially
with $M$. On the other hand, if the power series are expanded in
powers of $z/d$ for some $d\leq|\alpha|$, the condition numbers
become mild but the bound on the $\infty$-norm relative error becomes
useless when the computed vector is scaled by inverse powers of $d$.

In further support of the rule of thumb, we mention that triangular
matrices are typically ill-conditioned and their condition number
increases exponentially with the dimension of the matrix \cite{ViswanathTrefethen1998}.
For example, if all entries of a triangular matrix are independent
normal variables with mean $0$, the condition number increases at
the rate $2^{M}$, $M$ being the dimension of the matrix. 

An example where the accuracy does not deteriorate rapidly with $M$
in spite of the use of back substitution in an intermediate step occurs
in barycentric Hermite interpolation \cite{SadiqViswanath2011}. 

Triangular systems with back substitution come up in a natural way
if we want to determine the coefficients $b_{j}$ such that 
\[
b_{0}+b_{1}z+b_{2}z^{2}+\cdots=\frac{1}{a_{0}+a_{1}z+a_{2}z^{2}+\cdots}.
\]
However, the discussion here suggests that such a method will be inaccurate.
The fast Fourier transform (FFT) and contour integrals may have a
role in the numerically accurate invertion of series---see \cite{Bornemann2011}
for related ideas.

\section{Numerical Examples}

For simple choices of grid points, such as $z_{k}=0,\pm1,\pm2,\pm3,\pm4$,
Algorithms \ref{alg:fdweights-mlgrng} and \ref{alg:fdweights-partial-prod},
which are respectively based on the modified Lagrange formula (Section
3) and partial products (Section 4), as well as Fornberg's method
find the finite difference weights with errors that are very close
to machine precision. To compare the different methods, we must turn
to more complicated examples.

The Chebyshev points are defined by 
\[
z_{k}=\cos\left((k-1)\pi/(N-1)\right)=\sin\left(\pi(N-2k+1)/(N-1)\right)
\]
 for $k=1,\ldots,N$. We will look at the relative errors in the spectral
differentiation matrix of order $M$ for $M=2,4,8,16$ and $N=32,64,128,256,512$.
When Algorithm \ref{alg:fdweights-partial-prod} is employed, the
spectral differentiation matrix is computed as described in Section
5. 

The Chebyshev points are distributed over the interval $[-1,1]$.
The logarithmic capacity of an interval is one quarter its length,
which in this case is $1/2$. Therefore the Lagrange weights $w_{k}$
will be approximately of the order $1/2^{N}$. To prevent the possibility
of underflow for large $N$, the Chebyshev points are scaled to $2z_{k}$
and the resulting finite difference weights for the $M$-th derivative
are multiplied by $2^{-M}$. 

For reasons described in Section 2, the Chebyshev points are reordered.
The reordering we use is bit reversal. With $N$ being a power of
$2$ in our examples, the binary representation of $k$ (here $k$
is assumed to run from $1$ to $N-1$) can be reversed to map it to
a new position. The permutation induced by bit reversal is its own
inverse, which simplifies implementation. The reordering of the grid
points has the additional effect of making underflows less likely
\cite{SadiqViswanath2011}. For a discussion of various orderings
of Chebyshev points, see \cite{CalvettiReichel2003}. The figures
and plots here are given with the usual ordering of Chebyshev points.

\begin{figure}

\subfloat[Algorithm \prettyref{alg:fdweights-mlgrng}]{

\centering{}\includegraphics[scale=0.4]{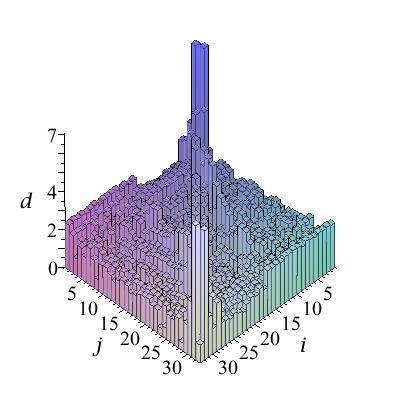}}\subfloat[Algorithm \prettyref{alg:fdweights-partial-prod}]{

\centering{}\includegraphics[scale=0.4]{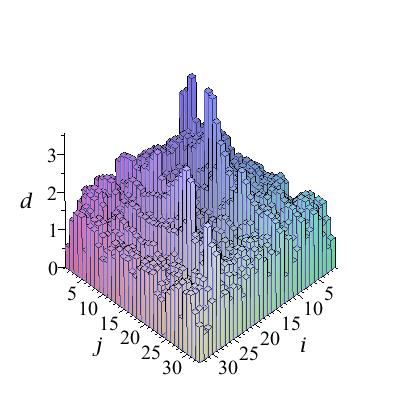}}\subfloat[Fornberg's algorithm]{\begin{centering}
\includegraphics[scale=0.4]{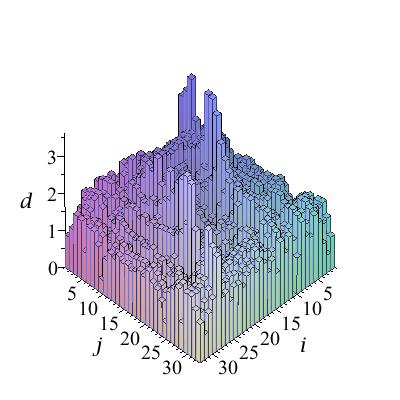}
\par\end{centering}

}\caption{Errors in the entries of the $32\times32$ Chebyshev differentiaton
matrix of order $M=8$. The vertical axis labeled $d$ shows the number
of digits of precision lost due to rounding errors. \label{fig:Errors-N32-M8}}

\end{figure}

Before turning to Figures \ref{fig:Errors-N32-M8} and \ref{fig:Variation-with-N},
which compare the numerical errors in different methods, we make an
important point. Even though the number of digits of precision lost
in the $32\times32$ differentiation matrix of order $M=8$ may be
just $3$, the errors in an $8$-th derivative evaluated using that
matrix will be much higher. Some entries of the $N\times N$ Chebyshev
differentiation matrix of order $M$ are $\mathcal{O}\left(N^{2M}\right)$.
Very large entries occur in the differentiation matrix and in exact
arithmetic an accurate derivative will be produced after delicate
cancelations during matrix-vector multiplication. In finite precision
arithmetic, the largeness of the entries implies that even rounding
errors in the entries that are of the order of machine epsilon are
sufficient to cause explosive errors in numerically computed derivatives.
The $512\times512$ Chebyshev differentiation matrix of order $16$
is useless even if every entry is computed with the maximum possible
$16$ digits in double precision arithmetic.

There are tricks for improving the accuracy of computed derivatives
or of entries of the Chebyshev differentiation matrix \cite{DonSolomonoff1995,WeidemanReddy2000}.
The so-called negative sum trick can be interpreted as a barycentric
formula for a derivative \cite{BaltenspergerTrummer2003,SchneiderWerner1986}.
This trick is limited to derivatives of the first order, and even
when $M=1$, it does not help when the differentiation matrix is inverted
in some form.

Our purpose here is to assess the accuracy with which the finite difference
weights are computed and we will stick to that purpose. From Figure
\ref{fig:Errors-N32-M8}, we see that Algorithm \ref{alg:fdweights-mlgrng},
which is based on the modified Lagrange formula, loses $7$ digits
for $N=32$ and $M=8$, while the other two methods lose only $3$
digits. There is a kind of flip symmetry in the errors shown in each
of the plots of that figure.

\begin{figure}

\subfloat[$M=2$]{\begin{centering}
\includegraphics[scale=0.3]{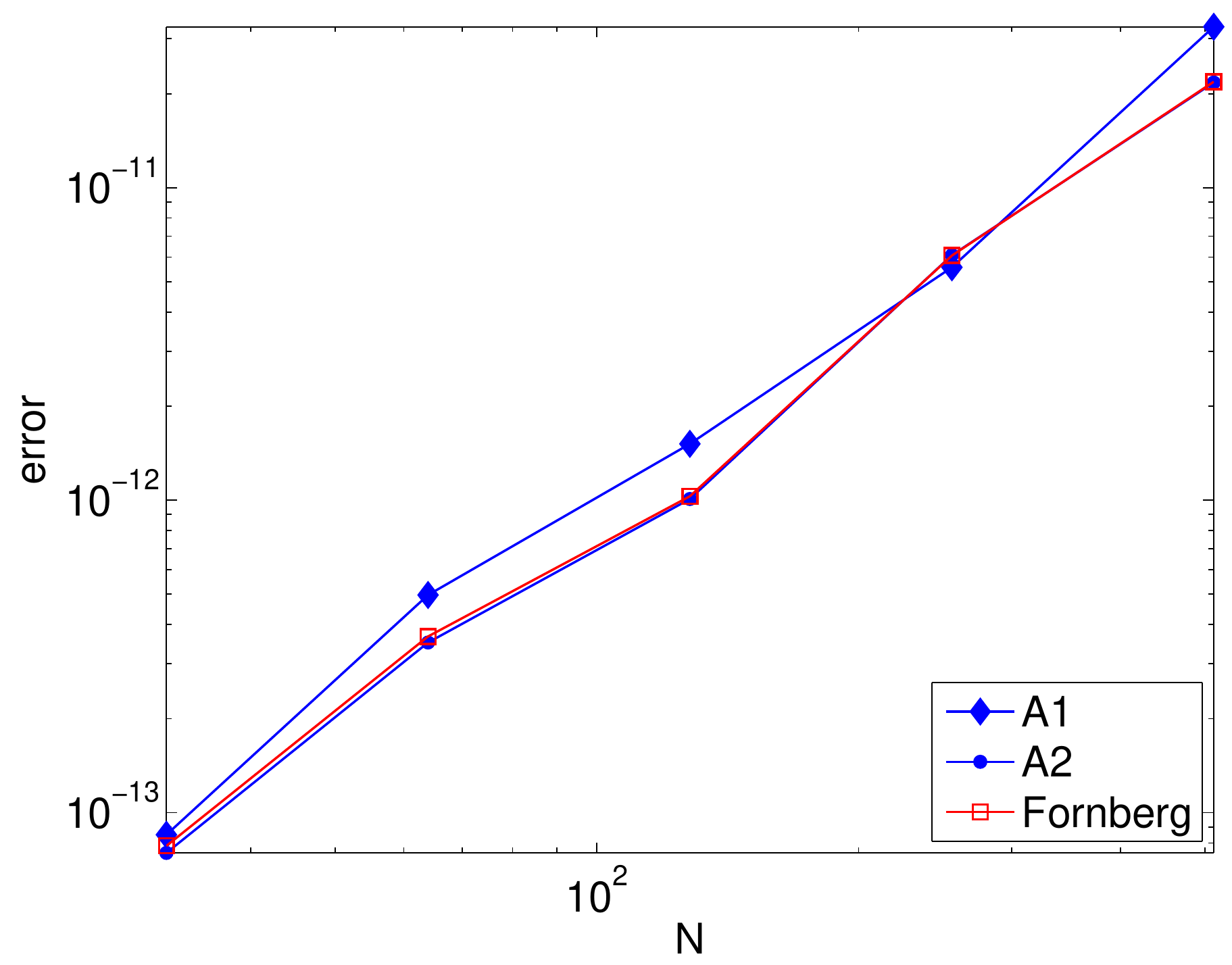}
\par\end{centering}

}\subfloat[$M=4$]{

\centering{}\includegraphics[scale=0.3]{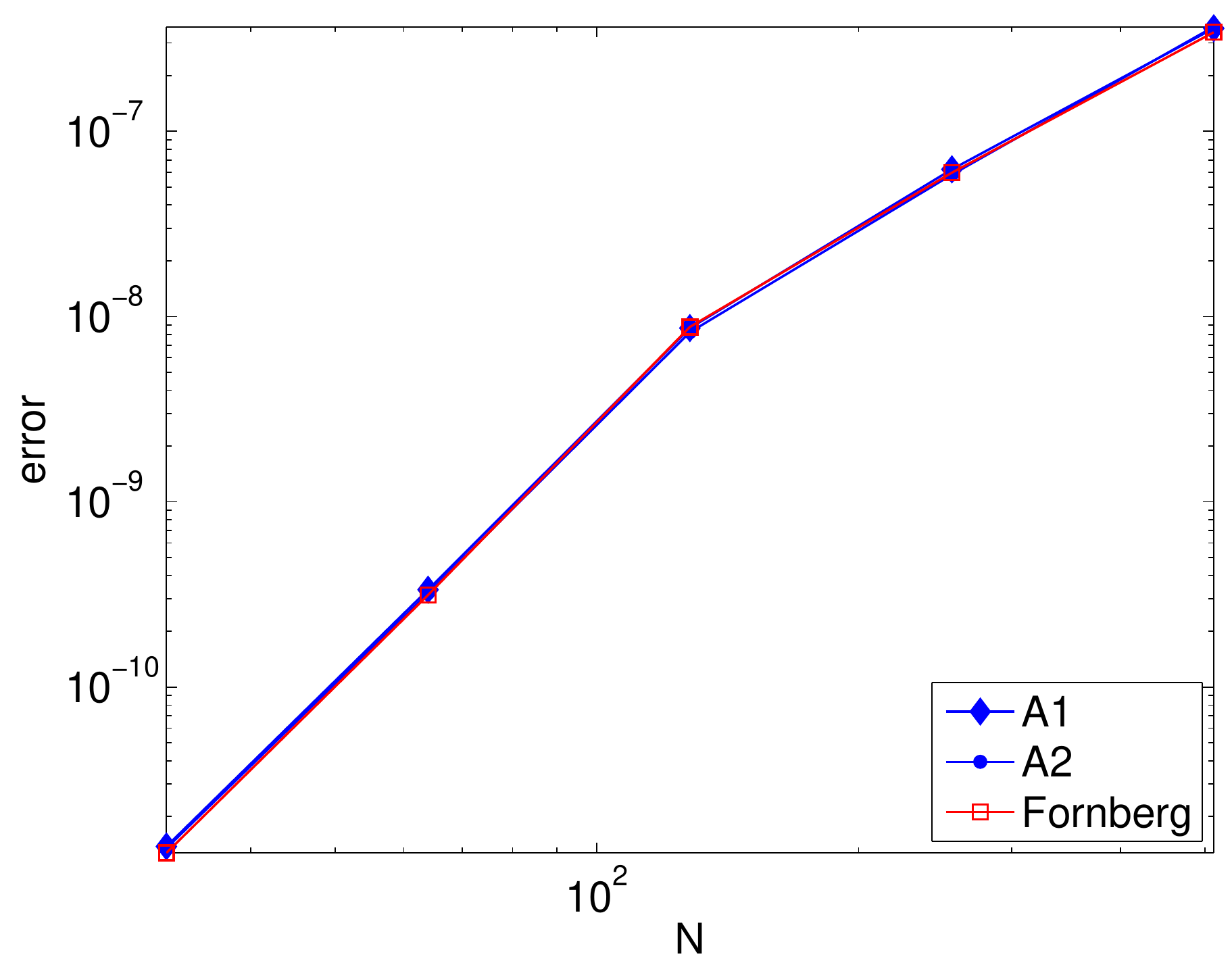}}

\subfloat[$M=8$]{\begin{centering}
\includegraphics[scale=0.3]{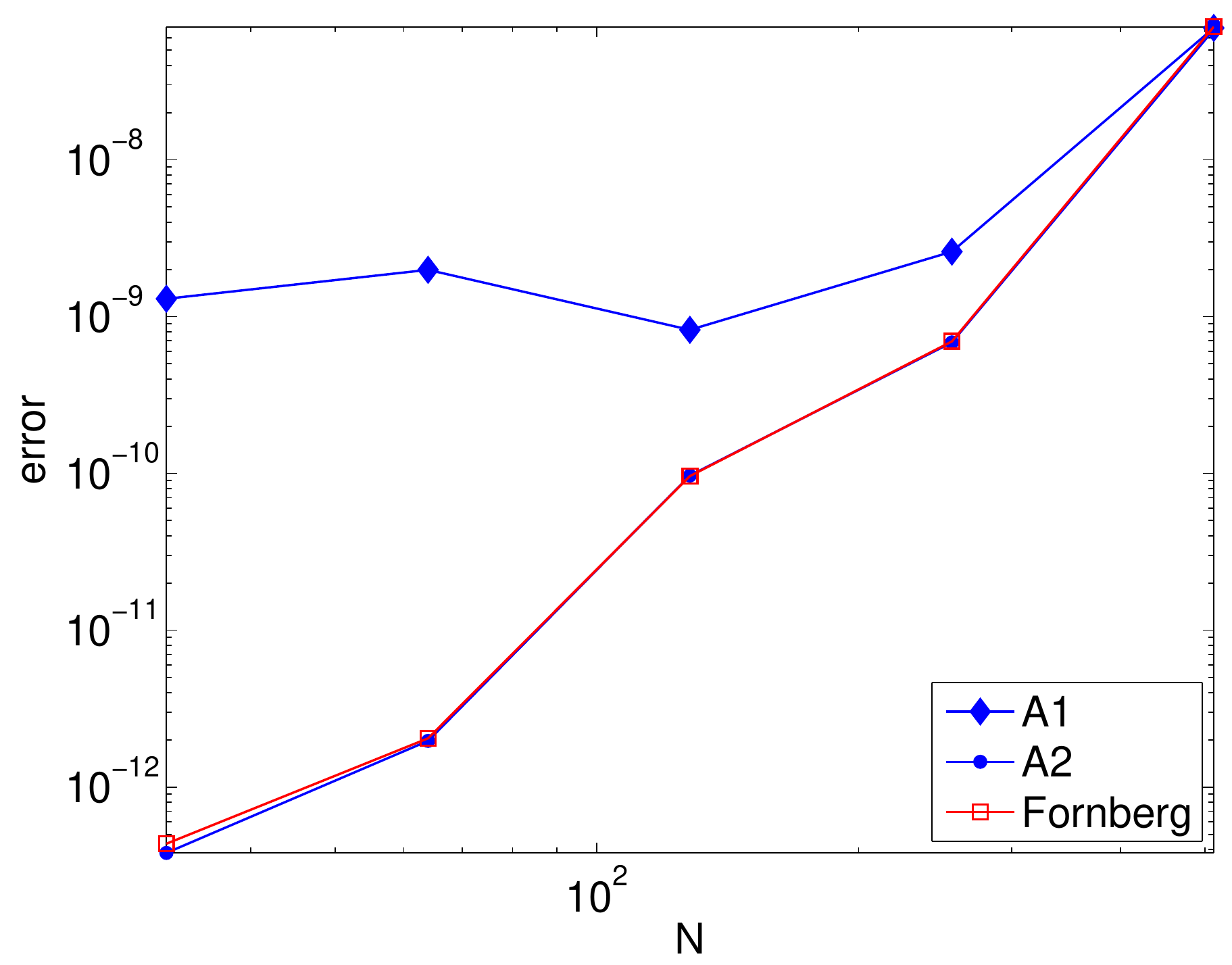}
\par\end{centering}

}\subfloat[$M=16$]{\begin{centering}
\includegraphics[scale=0.3]{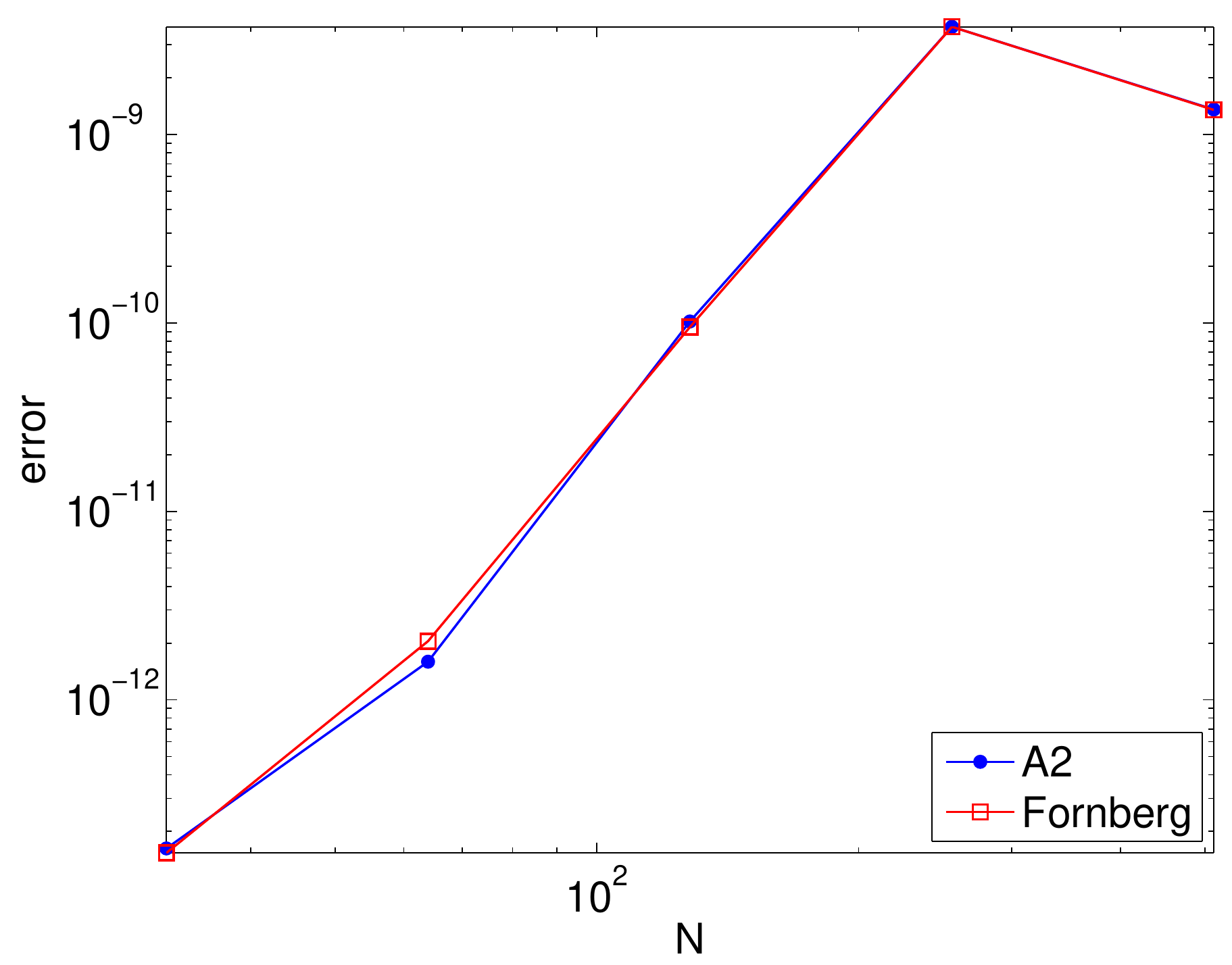}
\par\end{centering}

}\caption{Variation of the maximum relative error over the $N^{2}$ entries
of the $N\times N$ Chebyshev differentiation matrix. The order of
differentiation is $M$. In the legends, A1 and A2 stand for Algorithms
\ref{alg:fdweights-mlgrng} and \ref{alg:fdweights-partial-prod},
respectively.\label{fig:Variation-with-N}}

\end{figure}
 Figure \ref{fig:Variation-with-N} gives a more extensive report
of errors. All the errors were estimated using $50$ digit arithmetic
in MAPLE The errors were validated using $60$ digit arithmetic. From
the figure, we see that the algorithm based on partial products (A2
in the legends of the figure) is as accurate as Fornberg's method
in spite of using many fewer arithmetic operations. From the four
plots of Figure \ref{fig:Variation-with-N}, a surprise is that the
errors are smaller for $M=16$ than for $M=4$ or $M=8$. Why is that
the case? We are not certain of the answer. 

The errors in the algorithm based on the modified Lagrange formula
(A1 in the legends) is already noticeably larger for $M=8$. For $M=16$,
they are quite bad and of the order of $10^{5}$ (omitted from Figure
\ref{fig:Variation-with-N}D). Such a rapid deterioration in error
with increase in $M$ appears to validate the exponential instability
phenomenon hinted at in the previous section. 

The main finding of this section is that Algorithm \ref{alg:fdweights-partial-prod},
which is based on partial products, is as accurate as Fornberg's method
even though it uses fewer arithmetic operations.

\section{Superconvergence or boosted order of accuracy}

Let $z_{1},\ldots z_{N}$ be distinct grid points. Let 
\begin{equation}
f^{(m)}\left(0\right)\approx\frac{w_{1,m}f\left(hz_{1}\right)+\cdots+w_{N,m}f\left(hz_{N}\right)}{h^{m}}\label{eq:FD-with-h-repeat}
\end{equation}
be an approximation to the $m$-th derivative at $0$. We begin by
looking at the order of accuracy of this approximation. Here \prettyref{eq:FD-with-h}
is shown again as \prettyref{eq:FD-with-h-repeat} for convenience.
The order of the derivative $m$ is assumed to satisfy $m\leq N-1$.
The case $m=0$ corresponds to interpolation. The allowed values of
$m$ are from the set $\left\{ 1,2,\ldots,N-1\right\} $.
\begin{lem}
The finite difference formula \prettyref{eq:FD-with-h-repeat} has
an error of $\mathcal{O}\left(h^{N-m}\right)$ if and only if 
\[
\sum_{k=1}^{N}w_{k,m}x_{k}^{m}=m!\quad\text{and}\quad\sum_{k=1}^{N}w_{k,m}x_{k}^{n}=0
\]
for $n\in\left\{ 0,1,\ldots,N-1\right\} -\left\{ m\right\} $. The
function $f$ is assumed to be $N$ times continuously differentiable.\label{lem:The-finite-difference}\end{lem}
\begin{proof}
Assume that the weights $w_{k,m}$ satisfy the conditions given in
the lemma. The function $f(z)$ can be expanded using Taylor series
as $f(0)+f'(0)z+\cdots+f^{(N-1)}(0)z^{N-1}/(N-1)!+z^{N}g(z)$, where
$g(z)$is a continuous function. In particular, $g(z)$ is continuous
at $z=0$. If the Taylor expansion is substituted into the right hand
side of \prettyref{eq:FD-with-h-repeat} and the conditions satisfied
by the weights are used, we get the following expression:
\[
f^{(m)}(0)+h^{N-m}\left(w_{1,m}z_{1}^{N}g\left(hz_{1}\right)+\cdots+w_{N,m}z_{N}^{N}g\left(hz_{N}\right)\right).
\]
The coefficient of $h^{N-m}$ is bounded in the limit $h\rightarrow0$,
and therefore the error is $\mathcal{O}\left(h^{N-m}\right)$.

The necessity of the conditions on the weights $w_{k,m}$ is deduced
by applying the finite difference formula \prettyref{eq:FD-with-h-repeat}
to $f=1,z,\ldots,z^{N-1}$.
\end{proof}
The conditions on the weights in Lemma \ref{lem:The-finite-difference-2}
correspond to the following matrix system. 
\begin{equation}
\left(\begin{array}{cccc}
1 & 1 & \cdots & 1\\
z_{1} & z_{2} & \cdots & z_{N}\\
 &  & \cdots\\
z_{1}^{N-1} & z_{2}^{N-1} & \cdots & z_{N}^{N-1}
\end{array}\right)\left(\begin{array}{c}
w_{1,m}\\
w_{2,m}\\
\vdots\\
w_{N,m}
\end{array}\right)=m!e_{m}\label{eq:ooa-cond-weights}
\end{equation}
\label{pro:The-conditions-on}where $e_{m}$ is the unit vector with
its $m$-th entry equal to $1$. The matrix here is the transpose
of the well-known Gram or Vandermonde matrix. 

Newton and Lagrange interpolation are techniques for solving Vandermonde
systems. Newton interpolation is equivalent to an LU decomposition
of the Gram or Vandermonde matrix \cite{Davis1975}. Partly because
the matrix in \prettyref{eq:ooa-cond-weights} is the transpose of
the Gram or Vandermonde matrix, the interpolation techniques are not
directly applicable. 

The Gram or Vandermonde determinant equals $\prod_{1\leq i<j\leq N}\left(z_{j}-z_{i}\right)$
and is therefore nonsingular \cite{Davis1975}. Thus we have the following
theorem.
\begin{thm}
There exists a unique choice of weights $w_{k,m}$, $k=1,\ldots,N$,
such that the finite difference formula \prettyref{eq:FD-with-h-repeat}
has error $\mathcal{O}\left(h^{N-m}\right)$.
\end{thm}
This theorem is trivial and generally known. However, its clear formulation
is essential for developments that will follow. Our main interest
is in boosted order of accuracy. 
\begin{lem}
The finite difference formula \prettyref{eq:FD-with-h} has boosted
order of accuracy with an error of $\mathcal{O}\left(h^{N-m+b}\right)$,
where $b$ is a positive integer, if and only if the weights $w_{k,m}$
satisfy
\[
w_{1,m}z_{1}^{N-1+\beta}+\cdots+w_{N,m}z_{N}^{N-1+\beta}=0
\]
for $\beta=1,\ldots,b$ in addition to the conditions of Lemma \ref{lem:The-finite-difference}.\label{lem:The-finite-difference-2}\end{lem}
\begin{proof}
Similar to the proof of Lemma \ref{lem:The-finite-difference}.
\end{proof}
To derive conditions for boosted order of accuracy that do not involve
the weights, we introduce the following notation. By
\begin{equation}
\det\left(z_{1},z_{2}\ldots z_{N};n_{1},n_{2},\ldots,n_{N}\right)\label{eq:det-zjni}
\end{equation}
we denote the determinant of the $N\times N$ matrix whose $(i,j)$-th
entry is $z_{j}^{n_{i}}$. The transpose of the Vandermonde or Gram
determinant of the grid points, which occurs in \prettyref{eq:ooa-cond-weights},
is $\det(z_{1},\ldots,z_{N};0,\ldots N-1)$ in this notation.
\begin{thm}
Let $w_{k,m}$, $k=1,\ldots,N$, be the unique solution of \prettyref{eq:ooa-cond-weights}
so that the finite difference formula \prettyref{eq:FD-with-h-repeat}
has an order of accuracy that is at least $N-m$. The order of accuracy
is boosted by $b$, where $b$ is a positive integer, if and only
if 
\[
\det\left(z_{1},\ldots,z_{N};[0,1,\ldots,N-1,N-1+\beta]-m\right)=0
\]
for $\beta=1,\ldots,b$. Here $[0,1,\ldots,N-1,N-1+\beta]-m$ denotes
the sequence $0,1,\ldots N-1,N-1+\beta$ with $m$ deleted.\label{thm:boost-ooa-det}\end{thm}
\begin{proof}
First, assume the weights $w_{k,m}$ and the grid points $z_{k}$
to be real. The condition of Lemma \ref{lem:The-finite-difference-2}
requires that the row vector $W_{m}=\left[w_{1,m},\ldots,w_{N,m}\right]$
be orthogonal to 
\begin{equation}
\left[z_{1}^{N-1+\beta},\ldots,z_{N}^{N-1+\beta}\right].\label{eq:vector-beta}
\end{equation}
By \prettyref{eq:ooa-cond-weights}, $W_{m}$ is orthogonal to every
row of the Gram matrix except the $m$-th row. Since the Gram matrix
is non-singular, the rows of that matrix are a linearly independent
basis. Consequently, the $N-1$ dimensional space of vectors orthogonal
to $W_{m}$ is spanned by the rows of the Gram matrix with the $m$-th
row excepted. The vector \prettyref{eq:vector-beta} is orthogonal
to $W_{m}$ if and only if it lies in the span of the vectors 
\begin{equation}
\left[z_{1}^{n},\ldots,z_{N}^{n}\right]\quad n\in\left\{ 0,1\ldots N-1\right\} -\left\{ m\right\} .\label{eq:vectors-(N-1)}
\end{equation}
Thus the condition of Lemma \ref{lem:The-finite-difference-2} holds
if and only if the determinant of the $N\times N$ matrix whose first
$(N-1)$ rows are the vectors \prettyref{eq:vectors-(N-1)} and whose
last row is \prettyref{eq:vector-beta} vanishes as stated in the
theorem.

If the weights and the grid points are complex, the same argument
can be repeated after replacing the weights by their complex conjugates
in the definition of $W_{m}$.
\end{proof}
Theorem \ref{thm:boost-ooa-det} gives determinantal conditions for
boosted order of accuracy. We will cast those conditions into a more
tractable algebraic form. The following theorem gives the template
for the algebraic form into which the conditions of Theorem \ref{thm:boost-ooa-det}
will be cast.
\begin{thm}
If $n_{1},n_{2},\ldots,n_{N}$ are distinct positive integers, the
determinant \prettyref{eq:det-zjni} can be factorized as 
\[
\prod_{1\leq i<j\leq N}\left(z_{j}-z_{i}\right)\: S\left(z_{1},\ldots z_{N}\right),
\]
where $S(z_{1},\ldots z_{N})$ is a symmetric polynomial that is unchanged
when $z_{1},\ldots,z_{N}$ are permuted. All the coefficients of $S$
are integers.\label{thm:symm-poly-S}\end{thm}
\begin{proof}
We will work over $\mathbb{Q}$, the field of rational numbers. We
can think of the determinant \prettyref{eq:det-zjni} as a polynomial
in $z_{N}$ with coefficients in the field $\mathbb{Q}\left(z_{1},\ldots,z_{N-1}\right)$.
Since the determinant \prettyref{eq:det-zjni} vanishes, if $z_{N}$
is equal to any one of $z_{1},\ldots,z_{N-1}$, we have that the determinant
can be factorized as 
\[
\left(z_{N}-z_{1}\right)\left(z_{N}-z_{2}\right)\ldots\left(z_{N}-z_{N-1}\right)f
\]
where $f$ is an element of the field $\mathbb{Q}\left(z_{1},\ldots,z_{N-1}\right)$.
By Gauss's lemma, $f$ should in fact be an element of $\mathbb{Z}\left[z_{1},\ldots,z_{N-1}\right]$,
the ring of polynomials in $z_{1},\ldots,z_{N-1}$ with integer coefficients
(for Gauss's lemma, see Section 2.16 of \cite{JacobsonI} and in particular
the corollary at the end of that section). Now $f$ can be considered
as a polynomial in $z_{N-1}$ and factorized similarly, and so on,
until we get a factorization of the form shown in the theorem.

To prove that $S$ is symmetric, consider a transposition that switches
$z_{p}$ and $z_{q}$. The determinant \prettyref{eq:det-zjni} changes
sign by a familiar property of determinants. The product of all pairwise
differences $z_{j}-z_{i}$ also changes sign as may be easily verified
or as may be deduced by noting that the product is the Gram or Vandermonde
determinant. Therefore $S$ is unchanged by transpositions and is
a symmetric function. 
\end{proof}
For the determinants that arise as conditions for boosted order of
accuracy in Theorem \ref{thm:boost-ooa-det}, we describe a method
to compute the symmetric polynomial $S$ explicitly. The symmetric
polynomials that arise in Theorem \ref{thm:symm-poly-S} may well
have a connection to symmetric function theory. 

To begin with, let us consider the Gram determinant 
\begin{equation}
\det\left(z_{1},\ldots,z_{N},z_{N+1};0,\ldots,N-1,N\right).\label{eq:det-A1}
\end{equation}
This determinant is equal to 
\begin{equation}
\prod_{1\leq i<j\leq N}\left(z_{j}-z_{i}\right)\times\prod_{k=1}^{N}\left(z_{N+1}-z_{k}\right).\label{eq:exp-A1}
\end{equation}
See \cite[p. 25]{Davis1975}. By expanding \prettyref{eq:det-A1}
using the entries of the last column (each of these entries is a power
of $z_{N+1}$), we deduce that the coefficient of $z_{N+1}^{m}$ in
the expansion of \prettyref{eq:det-A1} is equal to 
\begin{equation}
(-1)^{N+m}\det(z_{1},\ldots,z_{N};[0,\ldots N-1,N]-m).\label{eq:boost-det-1}
\end{equation}
This determinant is the minor that corresponds to the entry $z_{N+1}^{m}$
in the expansion of \prettyref{eq:det-A1}. By inspecting \prettyref{eq:exp-A1},
we deduce that the coefficient of $z_{N+1}^{m}$ in that expression
is equal to 
\begin{equation}
\prod_{1\leq i<j\leq N}\left(z_{j}-z_{i}\right)\times(-1)^{N-m}S_{N-m},\label{eq:boost-cond-temp-1}
\end{equation}
where \texttt{
\[
S_{p}=\sum_{1\leq i_{1}<\cdots<i_{p}\leq N}z_{i_{1}}\ldots z_{i_{p}}.
\]
}Thus $S_{p}$ denotes the sum of all possible terms obtained by multiplying
$p$ of the grid points $z_{1},\ldots z_{N}$. For future use, we
introduce the notation $S_{p}^{+}$ for the sum of all possible terms
obtained by multiplying $p$ of the numbers $z_{1},\ldots,z_{N},z_{N+1}$.
\begin{thm}
The finite difference formula \prettyref{eq:FD-with-h-repeat} with
distinct grid points $z_{k}$ and weights $w_{k,m}$ that satisfy
\prettyref{eq:ooa-cond-weights} has an order of accuracy that is
boosted by $1$ if and only if $S_{N-m}=0$.\label{thm:boost-by-1}\end{thm}
\begin{proof}
The condition for a boost of $1$ is obtained by setting $\beta=1$
in Theorem \ref{thm:boost-ooa-det}. By equating \prettyref{eq:boost-det-1}
with \prettyref{eq:boost-cond-temp-1}, we get 
\begin{equation}
\det(z_{1},\ldots,z_{N};[0,\ldots N-1,N]-m)=\prod_{1\leq i<j\leq N}\left(z_{j}-z_{i}\right)\times S_{N-m}\label{eq:boost-eqn-1}
\end{equation}
Since the grid points are distinct, the determinant is zero if and
only if $S_{N-m}=0$.
\end{proof}
The corollary that follows covers all the popular cases that have
boosted order of accuracy.
\begin{cor}
If the grid points $z_{1},\ldots,z_{N}$ are symmetric about $0$
(in other words $z$ is a grid point if and only if $-z$ is a grid
point) and $N-m$ is odd, the order of accuracy is boosted by $1$.
\end{cor}
Although we have restricted $m$ to be in the set $\left\{ 1,2,\ldots,N-1\right\} $,
Theorems \ref{thm:boost-ooa-det} and \ref{thm:boost-by-1} hold for
the case $m=0$ as well. The case $m=0$ of \prettyref{eq:FD-with-h-repeat}
corresponds to interpolation. According to Theorem \ref{thm:boost-by-1},
the interpolation has boosted order of accuracy if and only if $S_{N}=0$
or one of the grid points is zero. Of course, the interpolant at zero
is exact if zero is one of the grid points. We do not consider the
case $m=0$ any further. 

To derive an algebraic condition for the order of accuracy to be boosted
by $2$, we apply the identity \prettyref{eq:boost-eqn-1} with grid
points $z_{1},\ldots,z_{N},z_{N+1}$ and rewrite it as follows.
\begin{multline*}
\det(z_{1},\ldots,z_{N},z_{N+1};[0,\ldots,N-1,N,N+1]-m)=\\
\prod_{1\leq i<j\leq N}\left(z_{j}-z_{i}\right)\times S_{N-m+1}^{+}\times\prod_{k=1}^{N}(z_{N+1}-z_{k}).
\end{multline*}
We equate the coefficients of $z_{N+1}^{N}$ to deduce that 
\begin{equation}
\det(z_{1},\ldots,z_{N};[0,\ldots N-1,N+1]-m)=\prod_{1\leq i<j\leq N}\left(z_{j}-z_{i}\right)\times\left(S_{1}S_{N-m}-S_{N-m+1}\right).\label{eq:boost-eqn-2}
\end{equation}
To obtain this identity, we assumed $m\geq1$ and used $S_{N-m+1}^{+}=S_{N-m+1}+z_{N+1}S_{N-m}$. 
\begin{lem}
The order of accuracy of the finite difference formula \prettyref{eq:FD-with-h-repeat}
is boosted by $2$ if and only if $S_{N-m}=0$ and $S_{N-m+1}=0$. \end{lem}
\begin{proof}
We already have the condition $S_{N-m}=0$ for the order of accuracy
to be boosted by $1$. By Theorem \ref{thm:boost-ooa-det}, the order
of accuracy is boosted by $2$ if and only if the determinant of \prettyref{eq:boost-eqn-2}
is zero as well. Since $S_{N-m}=0$, that is equivalent $S_{N-m+1}=0$.\end{proof}
\begin{thm}
The order of accuracy of the finite difference formula \prettyref{eq:FD-with-h-repeat}
for the $m$-th derivative can never be boosted by more than $1$
as long as the grid points are real. Here $m\geq1$.\end{thm}
\begin{proof}
By the preceding lemma, the grid points $z_{1},\ldots,z_{N}$ must
satisfy $S_{N-m}=0$ and $S_{N-m+1}=0$ for the order of accuracy
to be boosted by more than 1. First we consider $m=1$ and show that
$S_{N-1}=S_{N}=0$ is impossible. Since $N-1\geq m=1$, we must have
at least two grid points. Since $S_{N}=0$, at least one grid point
must be $0$. Since the grid points are distinct, no other grid point
is zero and $S_{N-1}\neq0$.

If $m\geq2$, let $r=N-m$. Then $r\geq1$. To show that $S_{r}=S_{r+1}=0$
is impossible, denote the elementary symmetric function formed by
adding all possible products of $r$ numbers out of $z_{1},\ldots,z_{N-1}$
by $s_{r}$. Then 
\begin{eqnarray*}
S_{r} & = & s_{r}+z_{N}s_{r-1}=0\\
S_{r+1} & = & s_{r+1}+z_{N}s_{r}=0.
\end{eqnarray*}
Here $s_{0}$ is taken to be $1$ as usual. Eliminating $z_{N}$,
we get $s_{r}^{2}=s_{r-1}s_{r+1}$.

Newton's inequality (see Theorem 144 on page 104 of \cite{HLP}) is
applied after noting that there are at least two numbers in the sequence
$z_{1},\ldots,z_{N-1}$ and that the numbers are all distinct. Newton's
inequality requires the numbers to be real. We get
\begin{eqnarray*}
\frac{s_{r}^{2}}{\binom{N-1}{r}^{2}} & > & \frac{s_{r-1}}{\binom{N-1}{r-1}}\frac{s_{r+1}}{\binom{N-1}{r+1}}\\
s_{r}^{2} & > & \left(\frac{N-r}{r}\frac{r+1}{N-r-1}\right)s_{r-1}s_{r+1}\\
s_{r}^{2} & > & s_{r-1}s_{r+1}
\end{eqnarray*}
It is impossible to have $s_{r}^{2}=s_{r-1}s_{r+1}$ or $S_{r}=S_{r+1}=0$
or $S_{N-m}=S_{N-m+1}=0$.
\end{proof}
If the grid points are complex, it may be possible to boost the order
of accuracy by more than 1. One may obtain formulas for the sequence
of determinants with $\beta=1,\ldots,b$ in Theorem \ref{thm:boost-ooa-det}.
We have already covered the case with $\beta=1$ in \prettyref{eq:boost-eqn-1}
and the case with $\beta=2$ in \prettyref{eq:boost-eqn-2}. To illustrate
the general procedure, we show how to get a formula for the determinant
of Theorem \ref{thm:boost-ooa-det} with $\beta=3$. We write down
the identity \prettyref{eq:boost-eqn-2} using the grid points $z_{1},\ldots,z_{N},z_{N+1}$
and replace $N$ by $N+1$. 

\begin{multline*}
\det(z_{1},\ldots,z_{N},z_{N+1};[0,\ldots N-1,N,N+2]-m)=\\
\prod_{1\leq i<j\leq N}\left(z_{j}-z_{i}\right)\times\left(S_{1}^{+}S_{N-m+1}^{+}-S_{N-m+2}^{+}\right)\times\prod_{k=1}^{N}(z_{N+1}-z_{k}).
\end{multline*}
We use $S_{1}^{+}=S_{1}+z_{N+1}$, $S_{N-m+1}^{+}=S_{N-m+1}+z_{N+1}S_{N-m}$,
and $S_{N-m+2}^{+}=S_{N-m+2}+z_{N+1}S_{N-m+1}$, and equate coefficients
of $z_{N+1}^{N}$ to get
\begin{multline*}
\det(z_{1},\ldots,z_{N};[0,\ldots N-1,N+2]-m)=\\
\prod_{1\leq i<j\leq N}\left(z_{j}-z_{i}\right)\times\left(S_{N-m+2}-S_{N-m+1}S_{1}+S_{N-m}S_{1}^{2}-S_{N-m}S_{2}\right).
\end{multline*}
This is the determinant with $\beta=3$ in Theorem \ref{thm:boost-ooa-det}.
It gets cumbersome to go on like this. However, we notice that the
condition for the determinants with $\beta=1,2,3$ to be zero is $S_{N-m}=S_{N-m+1}=S_{N-m+2}=0$.
Here a simple pattern is evident. 

To prove this pattern, we assume that the determinant of Theorem \ref{thm:boost-ooa-det}
with $\beta=r$ is of the form given by Theorem \ref{thm:symm-poly-S}
with 
\[
S=S_{N-m+r-1}+\text{more terms}
\]
where each term other than the first has a factor that is one of $S_{N-m},\ldots,S_{N-m+r-2}$.
We pass to the case $\beta=r+1$ using the grid points $z_{1},\ldots,z_{N},z_{N+1}$
as illustrated above. Then it is easy to see that the form of $S$
for $\beta=r+1$ is 
\[
S=S_{N-m+r}+\text{more terms}
\]
where each term other than the first has a factor that is one of $S_{N-m},\ldots,S_{N-m+r-1}$.
If the determinants with $\beta=1,\ldots,r$ in Theorem \ref{thm:boost-ooa-det}
are zero, the additional condition that must be satisfied by the grid
points for the determinant with $\beta=r+1$ to be zero is $S_{N-m+r}=0$.
\begin{thm}
The order of accuracy of the finite difference formula \prettyref{eq:FD-with-h-repeat}
for the $m$-th derivative is boosted by $b$ if and only if $S_{N-m}=S_{N-m+1}=\cdots=S_{N-m+b-1}=0$.
Even with complex grid points, the order of accuracy can never be
boosted by more than $m$.\end{thm}
\begin{proof}
The first part of the theorem was proved by the calculations that
preceded its statement. To prove the second part, suppose that the
order of accuracy is boosted by $m+1$. Then we must have $S_{N}=0$
which means at least one of the grid points is zero. Since no other
grid point can be zero, we must have $S_{N-1}\neq0$, which is a contradiction.
\end{proof}
\begin{algorithm}
\begin{algorithmic}
\State Input: Grid points $z_1,\ldots,z_N$ all of which are real.
\State Input: Order of derivative $m$ with $1\leq m\leq N-1$.
\State Input: Weights $w_{1,m},w_{2,m},\ldots,w_{N,m}$ in the finite difference 
formula for $f^{(m)}(0)$.
\State Comment: $w_{k,m}$ are computed using Algorithm 2. 
\State Input: Tolerance $\tau$
\State $S_{N-m} = \sum_{1\leq i_1<\cdots<i_{N-m}\leq N} z_{i_1}\ldots z_{i_{N-m}}$.
\State $T_{N-m} = \sum_{1\leq i_1<\cdots<i_{N-m}\leq N} \bigl|z_{i_1}\ldots z_{i_{N-m}}\bigr|$.
\If{$\bigl|S_{N-m}\bigr|<\tau T_{N-m}$}
\State $r = N-m+1$.
\Else
\State $r = N-m$.
\EndIf
\State $C=\sum_{k=1}^{k=N} w_{k,m} z_k^{r+m}$.
\State Leading error term of (1.2): $C\frac{f^{(r+m)}(0)}{(r+m)!}h^r$.
\end{algorithmic}

\centering{}\caption{Order of Accuracy and Error Constant\label{alg:Order-of-Accuracy}}
\end{algorithm}
Algorithm \ref{alg:Order-of-Accuracy} uses the results of this section
to determine the order of accuracy and the leading error term in the
case of real grid points. We suspect that the error is exactly equal
to $C\frac{f^{(r+m)}(\zeta)}{(r+m)!}h^{r}$ for some point $\zeta$
in an interval that includes $0$ and all the grid points.

\section{Conclusion}

Algorithm \ref{alg:fdweights-partial-prod} uses partial products
of binomials of the type $\prod_{j=1}^{j=k}(z-z_{j})$ to compute
finite different weights with good accuracy. A C++ implementation
of the method will be posted on the internet. This method lends itself
to the efficient computation of spectral differentiation matrices
as described in Section 5.

Many finite difference formulas, such as the centered difference formulas
for the first and second derivatives, have an order of accuracy which
is higher than the typical by $1$. In Section 8, we proved theorems
which characterize superconvergence or boosted order of accuracy of
finite difference formulas completely.

\section{Acknowledgements}

The authors thank John Boyd, Nick Trefethen and Oleg Zikanov for useful
discussions. 

\bibliographystyle{plain}
\bibliography{references}

\end{document}